\documentclass{article}

\usepackage{amsmath,amsfonts,amsthm,amssymb,amscd,enumerate}

\usepackage{amentashortcuts2}
\usepackage{amentastyle2}
\usepackage[draft=false, colorlinks=true]{hyperref}

\usepackage{authblk}

\usepackage{comment}

\usepackage{esint}

\theoremstyle{plain}
\newtheorem{claim}[thm]{Claim}

\newcommand\blfootnote[1]{%
  \begingroup
  \renewcommand\thefootnote{}\footnote{#1}%
  \addtocounter{footnote}{-1}%
  \endgroup
}

\title{Non-uniformly local tent spaces}
\author[1]{Alex Amenta\thanks{amenta@fastmail.fm}}
\author[2]{Mikko Kemppainen\thanks{mikko.k.kemppainen@helsinki.fi}}
\affil[1]{Mathematical Sciences Institute, Australian National University}
\affil[2]{Departamento de Matem\'aticas, Universidad Aut\'onoma de Madrid}
\date{}
\begin{document}

\maketitle
%\tableofcontents


\begin{abstract}
  We develop a theory of `non-uniformly local' tent spaces on metric measure spaces.
  As our main result, we give a remarkably simple proof of the atomic decomposition.
\end{abstract}


\section{Introduction}

The theory of global tent spaces on Euclidean space was first considered by Coifman, Meyer, and Stein \cite{CMS85}, and has since become a central framework for understanding Hardy spaces defined by square functions.
Upon replacing Euclidean space with a doubling metric measure space, the theory is largely unchanged.\footnote{Details of this generalisation can be found in \cite{aA13}, although this was known to harmonic analysts for some time.}

Tent spaces on Riemannian manifolds with doubling volume measure were used by Auscher, McIntosh, and Russ in \cite{AMR08}, where a `first order approach' to Hardy spaces associated with the Laplacian $-\Delta$ (or more accurately, the corresponding Hodge--Dirac operator) was investigated.
A corresponding \emph{local} tent space theory, now on manifolds with \emph{exponentially locally doubling} volume measure, was considered by Carbonaro, McIntosh, and Morris \cite{CMM13}, with applications to operators such as $-\Delta + a$ for $a > 0$.
The locality arises from the `spectral gap' between $0$ and $\sigma (-\Delta + a) \subset [a,\infty)$ and means that the relevant information of a function can be captured from small time diffusion,
which in turn allows one to exploit the locally doubling nature of the manifold under investigation.
Hence the related tent spaces consist of functions of space-time variables $(y,t)$ with $0 < t < 1$ instead of $0 < t < \infty$.

The motivation for non-uniformly local tent spaces comes from the setting of Gaussian harmonic analysis, in which one considers the \emph{Ornstein--Uhlenbeck operator} ${L = -\Delta + x\cdot \nabla}$ on $\RR^n$ equipped with the usual Euclidean distance and the Gaussian measure $d\gg (x) = (2\pi)^{-n/2} e^{-|x|^2 / 2} \, dx$.
Here $\sigma (L) = \{ 0,1,2,\ldots \}$, but despite the evident spectral gap, one cannot make use of a uniformly local tent space because the rapidly decaying measure $\gg$ is non-doubling. 
This was remedied by Maas, van Neerven, and Portal \cite{MvNP12}, who defined the `Gaussian tent spaces' $\mf{t}^p(\gg)$ to consist of functions on the region $D = \{ (y,t) \in\RR^n \times (0,\infty) : t < m(y) \}$.
Here $m(y) = \min (1, |y|^{-1})$ is the \emph{admissibility function} of Mauceri and Meda \cite{MM07}, who showed that $\gg$ is doubling on the family of `admissible balls' $B(x,t)$ with $t\leq m(x)$.
In \cite{pP13}, Portal then defined the `Gaussian Hardy space' $\mf{h}^1(\gg)$ using the conical square function
\begin{equation*}
  Su(x) = \left( \int_0^{2m(x)} \fint_{B(x,t)} |t\nabla e^{-t^2L}u(y)|^2 \,d\gg (y) 
    \frac{dt}{t} \right)^{1/2} ,
\end{equation*}
and showed that the Riesz transform $\nabla L^{-1/2}$ is bounded from $\mf{h}^1(\gg)$ to $L^1(\gg)$.
This relied on the atomic decomposition on $\mf{t}^1(\gg)$, which was established in \cite{MvNP12}, along with a square function estimate from \cite{MvNP11}.
The Gaussian Hardy space is also known to interpolate with $L^2(\gg)$, in the sense that $[\mf{h}^1(\gg) , L^2(\gg)]_\theta = L^p(\gg)$ for $1/p = 1 - \theta / 2$ \cite{PPpc}.\footnote{Note that dimension-independent boundedness of $\nabla L^{-1/2}$ on $L^p(\gg)$ for $1 < p < \infty$ is a classical result of Meyer \cite{pM84}.}

Our long-term aim is to generalise this theory to the setting where, given an appropriate `potential function' $\phi$ on a Riemannian manifold $X$ (or some more general space) with volume measure $\gm$, one considers the \emph{Witten Laplacian} $L = -\gD + \nabla \phi \cdot \nabla$ equipped with the geodesic distance and the measure $d\gg = e^{-\phi}d\gm$.
An admissibility function can then be defined by $m(x) = \min (1,|\nabla\phi (x)|^{-1})$, with a suitable interpretation of $\nabla$ if $\gfv$ is not differentiable, and the setting of Gaussian harmonic analysis is recovered by taking $X = \RR^n$ and $\phi (x) = \frac{n}{2}\log (2\pi) + \frac{|x|^2}{2}$.
The Riesz transform associated with the Witten Laplacian has been studied for instance by Bakry in \cite{dB87}, where $L^p(\gg)$ boundedness for $1 < p < \infty$ is proven under a $\phi$-related curvature assumption.

In this article we define and study the corresponding local tent spaces $\mf{t}^{p,q}(\gg)$. 
Our main result is the atomic decomposition Theorem \ref{atomic}.
This allows us to identify the dual of $\mf{t}^{1,q}(\gg)$ with the local tent space $\mf{t}^{\infty,q^\prime}(\gg)$, and to show that the local tent spaces form a complex interpolation scale. In Appendix \ref{cone} we prove a `cone covering lemma'
for non-negatively curved Riemannian manifolds. It gives a stronger version of Lemma \ref{pointwise2} that is
applicable also in the vector-valued theory of tent spaces (see \cite{mK11,mK14}).

A different approach to Gaussian Hardy spaces was initiated in \cite{MM07},
where the \emph{atomic} Hardy space $H^1(\gg)$ was introduced. This theory has also been 
extended to certain metric measure spaces (see \cite{CMM09,CMM10}). 
While many interesting singular integral operators, such as
imaginary powers of the Ornstein--Uhlenbeck operator, have been shown to act boundedly from $H^1(\gg)$
to $L^1(\gamma)$ (see \cite{MM07}), it should be noted that this is not the case for the Riesz transform (see
\cite{MMS12}). This marks the crucial difference between the atomic Hardy space $H^1(\gg)$ and $\mf{h}^1(\gg)$.

\addtocontents{toc}{\protect\setcounter{tocdepth}{1}}
\subsection*{Acknowledgements}
The first author acknowledges the financial support from the Australian Research Council Discovery grant DP120103692.
The second author acknowledges the financial support from the V\"ais\"al\"a Foundation and from the Academy of Finland through the project \emph{Stochastic and harmonic analysis: interactions and applications} (133264). He is grateful for the hospitality of the Mathematical Sciences Institute at the Australian National University during his stay.
\addtocontents{toc}{\protect\setcounter{tocdepth}{2}}

\section{Weighted measures and admissible balls}\label{basicsetting}

We begin by formulating the abstract framework in which we develop our theory.
Let $(X,d,\gm)$ be a metric measure space: that is, a metric space $(X,d)$ equipped with a Borel measure $\gm$.
We assume that every ball $B\subset X$ comes with a given center $c_B$ and a radius $r_B > 0$, and that 
the volume $\gm(B)$ is finite and nonzero.
Furthermore, we assume that the metric space $(X,d)$ is geometrically doubling: that is, we assume that there exists a natural number $N \geq 1$ such that for every ball $B \subset X$ of radius $r_B$, there exist at most $N$ mutually disjoint balls of radius $r_B/2$ contained in $B$.

Given a measurable real-valued function $\gfv$ on $X$, we consider the weighted measure 
\begin{equation*}
  d\gg (x) := e^{-\gfv (x)} \; d\gm(x).
\end{equation*}
Furthermore, we fix a function $\map{m}{X}{(0,\infty)}$, which we call an \emph{admissibility function}.
For every $\ga > 0$, this defines the family of \emph{admissible balls}
\begin{equation*}
  \mc{B}_\ga := \{ B\subset X : 0 < r_B \leq \ga m(c_B) \} .
\end{equation*}
These objects are required to satisfy the following doubling condition:

\begin{description}
\item[(A)] For every $\ga > 0$, $\gg$ is doubling on $\mc{B}_\ga$, in the sense that there exists a constant $C_\ga \geq 1$ such that for all $\ga$-admissible balls $B \in \mc{B}_\ga$,
  \begin{equation*}
    \gg(2B) \leq C_\ga \gg (B).
  \end{equation*}
\end{description}
Here and in what follows, we write $\gl B = B(c_B , \gl r_B)$ for the expansion of a ball $B$ by $\lambda \geq 1$. 

\begin{rmk}
\label{doublingconstant}
    Condition (A) implies that for every $\ga > 0$ and every $\gl \geq 1$, there exists
    a constant $C_{\ga,\gl} \geq 1$ such that for all $\ga$-admissible balls $B \in \mc{B}_\ga$,
    \begin{equation}\label{alphalambda}
      \gg (\gl B) \leq C_{\ga,\gl} \gg (B) .
    \end{equation}
\end{rmk}

We now describe two classes of examples of $\gfv$ and $m$.

\begin{example}[Distance functions]\label{distfn}
  Assume that the underlying measure $\gm$ is doubling, let $\gO \subset X$ be a measurable set of `origins', and let $a,a^\prime > 0$.
  Define $\gfv$ by
  \begin{equation*}
    \gfv(x) := a + a' \dist (x,\gO)^2.
  \end{equation*} 
  An admissibility function can then be defined by
  \begin{equation*}
    m(x) = \min \left(1, \frac{1}{\dist(x,\gO)} \right) .
  \end{equation*}
  Taking $X$ to be $\RR^n$ (equipped with the usual Euclidean distance and Lebesgue measure), $\gO = \{0\}$, and $(a,a^\prime) = (n\log(2\gp)/2,1/2)$, we recover the setting of Gaussian harmonic analysis.

\begin{comment}
  \begin{claim}
    Condition (A1') is satisfied with $c_\ga^\prime = \ga + 1$.
  \end{claim}

  \begin{proof}
    First, if $\dist(y,\gO) \leq 1$ or $\dist(y,\gO) \leq \dist(x,\gO)$, then it is clear that $m(x)\leq m(y)$. 
    Second, if $1 < \dist(y,\gO) \leq 1 + \ga$, then $m(y) \geq (1+\ga )^{-1}$, and so we have $m(x) \leq 1 \leq (1+\ga)m(y)$.
    Finally, consider the case when $\dist(y,\gO) > 1+\ga$ and
    $\dist(y,\gO) \geq \dist(x,\gO)$. Now
    \begin{align*}
      m(x) \leq \frac{1}{\dist(x,\gO)}
      \leq \frac{1}{\dist(y,\gO) - d(x,y)}
      \leq \frac{1}{\dist(y,\gO) - \ga}
      &= \frac{m(y)}{1-\ga m(y)} \\
      &\leq \frac{m(y)}{1 - \frac{\ga}{1 + \ga}} \\
      &= (1+\ga)m(y),
    \end{align*}
    where in the penultimate step we used the assumption that $m(y) < (1+\alpha)^{-1}$.
  \end{proof}
\end{comment}

  \begin{claim}
    Condition (A) is satisfied with $C_\ga = D_\gm e^{a^\prime \ga(5\ga + 6)}$, where $D_\gm$ is the doubling constant of the underlying measure $\gm$.
  \end{claim} 

  \begin{proof}
    Since $\gm$ is doubling, it suffices to show that for every $\ga$-admissible ball $B \in \mc{B}_\ga$ we have 
    \begin{equation}
      \label{localdoubling}
      \begin{cases}
        e^{-\gfv(x)} \leq C_\ga' e^{-\gfv(c_B)}  \quad \text{when $x\in 2B$, and} \\
        e^{-\gfv(x)} \geq C_\ga'' e^{-\gfv(c_B)}  \quad \text{when } x\in B .
      \end{cases}
    \end{equation}	
    Indeed, this would imply that
    \begin{equation*}
      \gg (2B) = \int_{2B} e^{-\gfv (x)} \; d\gm(x) \leq C_\ga' \gm(2B) e^{-\gfv (c_B)}
    \end{equation*}
    and 
    \begin{equation*}
      \gg (B) = \int_B e^{-\gfv (x)} \; d\gm(x) \geq C_\ga'' \gm(B) e^{-\gfv (c_B)},
    \end{equation*}	
    so that
    \begin{equation*}
      \frac{\gg (2B)}{\gg (B)} 
      \leq \frac{C_\ga'}{C_\ga''} \frac{\gm(2B)}{\gm(B)} \leq C_\ga := D_\gm \frac{C_\ga^\prime}{C_\ga^{\prime\prime}}.
    \end{equation*}
    
    To see that the first inequality in \eqref{localdoubling} holds with $C_\ga^\prime = e^{4a^\prime \ga(\ga+1)}$, observe that if $x\in 2B$, then
    \begin{equation*}
      \dist(c_B,\gO) \leq 2\ga m(x) + \dist(x,\gO) .
    \end{equation*}
    Indeed, if $\dist(c_B,\gO) \geq \dist(x,\gO)$, then $m(c_B)\leq m(x)$, and so
    \begin{equation*}
      \dist(c_B,\gO) \leq d(c_B,x) + \dist(x,\gO) \leq 2\ga m(c_B) + \dist(x,\gO) \leq 2\ga m(x) + \dist(x,\gO).
    \end{equation*}
    Consequently we have
    \begin{equation*}
      \dist(c_B,\gO)^2 \leq 4\ga m(x)^2 + 4\ga m(x) \dist(x,\gO) + \dist(x,\gO)^2
      \leq 4\ga^2 + 4\ga + \dist(x,\gO)^2,
    \end{equation*}
    and so
    \begin{equation*}
      e^{-a'\dist(x,\gO)^2} \leq e^{4a'\ga(\ga + 1)} e^{-a'\dist(c_B,\gO)^2}.
    \end{equation*}
    
    Similarly, the second inequality in \eqref{localdoubling} with $C_\ga^{\prime\prime} = e^{-a^\prime \ga(\ga+2)}$ follows after noting that
    if $x\in B$, then
    \begin{equation*}
      \dist(x,\gO) \leq d(x,c_B) + \dist(c_B,\gO) \leq \ga m(c_B) + \dist(c_B,\gO).
    \end{equation*}
    Thus
    \begin{equation*}
      \dist(x,\gO)^2 \leq \ga^2 + 2\ga + \dist(c_B,\gO)^2
    \end{equation*}
    and
    \begin{equation*}
      e^{-a'\dist(x,\gO)^2} \geq e^{-a^\prime \ga(\ga + 2)} e^{-a'\dist(c_B,\gO)^2}.
    \end{equation*}
    Putting these estimates together, we have
    \begin{equation*}
      C_\ga = D_\gm e^{4a^\prime \ga(\ga+1)}e^{a^\prime \ga(\ga+2)} = D_\gm e^{a^\prime \ga(5\ga + 6)}
    \end{equation*}
    as claimed.
  \end{proof}
\end{example}

\begin{example}[$C^2$ potentials]\label{c2pot}
  In this example, let $(X,g)$ be a connected Riemannian manifold ($C^2$ is sufficient) with doubling volume measure, let $\gfv \in C^2(X)$, and assume that the following condition is satisfied:
  \begin{description}
  \item[(B)] there exists a constant $M > 0$ such that for every unit speed geodesic $\map{\gr}{[0,\ell]}{X}$, we have
    \begin{equation}\label{derivativetest}
      |(\gfv \circ \gr)^{\prime\prime}(t)| \leq M|(\gfv \circ \gr)^{\prime}(t)|
    \end{equation}
    for all $t \in (0,\ell)$ such that $|(\gfv \circ \gr)^{\prime}(t)| > 1$.
  \end{description}
  Alternatively, we can assume the following inequivalent condition, which is neater but generally harder to verify:
  \begin{description}
  \item[(H)] there exists a constant $M > 0$ such that
    \begin{equation}\label{hesscondn}
      \norm{\Hess \gfv(x)} \leq M|\nabla \gfv(x)|
    \end{equation}
    for all $x \in X$ such that $|\nabla \gfv(x)| > 1$.
  \end{description}
  Note that (B) can be interpreted as a one-dimensional version of (H); indeed, when $X$ is one-dimensional, both conditions are equivalent.

  If either of the above conditions are satisfied, we define an admissibility function by
  \begin{equation*}
    m(x) := \min\left( 1, \frac{1}{|\nabla \gfv(x)|} \right)
  \end{equation*}
  for $x \in X$, with $m(x) := 1$ when $|\nabla \gfv(x)| = 0$.

  \begin{claim}\label{a12}
    If $d(x,y) \leq \ga$ then $m(x) \leq e^{M\ga} m(y)$.  
  \end{claim}

  \begin{proof}
    Here we assume condition (H); the proof under assumption (B) requires only a simple modification.
    
    Given $\ge > 0$, we first take a continuous arclength-parametrised path $\map{\gr}{[0,d(x,y)+\ge]}{X}$ connecting $x$ to $y$ (we may take $\ge = 0$ when X is complete, and the argument is slightly simpler in this case).
    Since $\gfv$ is twice continuously differentiable, the function $m_\gr := m \circ \gr$ is absolutely continuous on $[0,d(x,y)]$, and hence differentiable almost everywhere on this interval.
    We compute the derivative of $m_\gr(t)$ whenever $m_\gr$ is differentiable.
    If $t$ is such that $|\nabla \gfv(\gr(t))| \leq 1$ in a neighbourhood of $t$, then $\partial_t m_\gr(t) = 0$.
    If $t$ is such that $|\nabla \gfv(\gr(t))| > 1$ in a neighbourhood of $t$, then
    \begin{equation*}
      \partial_t m_\gr(t) = \partial_t (|\nabla \gfv(\gr(t))|^{-1})
      = \frac{-\partial_t |\nabla \gfv(\gr(t))|}{|\nabla \gfv(\gr(t))|^2}.
    \end{equation*}
    Using the estimate
    \begin{equation*}
      |\partial_t |\nabla \gfv(\gr(t))|| \leq \norm{\Hess \gfv(\gr(t))}		
    \end{equation*}
    along with assumption (H), we find that
    \begin{equation*}
      |\partial_t m_\gr(t)| \leq \frac{\norm{\Hess \gfv(\gr(t))}}{|\nabla \gfv(\gr(t))|^2} \leq \frac{M}{|\nabla \gfv(\gr(t))|}
    \end{equation*}
    for all $t$ such that $m_\gr(t)$ is differentiable.
    
    Since $m_\gr(t)$ is differentiable almost everywhere, we have
    \begin{align*}
      |\log m_\gr(d(x,y)+\ge) - \log m_\gr(0)| &\leq \sup_{0 < t < d(x,y)+\ge} |\partial_t \log m_\gr(t)|(d(x,y)+\ge) \\
      &\leq \sup_{0<t<d(x,y)+\ge} |\partial_t \log m_\gr(t)|(\ga+\ge),
    \end{align*}
    where the supremum is taken over all $t \in (0,d(x,y)+\ge)$ such that $m_\gr(t)$ is differentiable.
    Note that
    \begin{equation*}
      |\partial_t \log m_\gr(t)| = \frac{|\partial_t m_\gr(t)|}{|m_\gr(t)|},
    \end{equation*}
    and so by the estimate above we have that
    \begin{equation*}
      |\partial_t \log m_\gr(t)| \leq \frac{M}{|\nabla \gfv(\gr(t))|} |\nabla \gfv(\gr(t))| = M.
    \end{equation*}
    Therefore
    \begin{equation*}
      |\log m_\gr(d(x,y)+\ge) - \log m_\gr(0)| \leq M(\ga+\ge),
    \end{equation*}
    and so
    \begin{equation*}
      e^{|\log (m(y) / m(x))|} \leq e^{M(\ga+\ge)} =: c_\ga^\prime e^{M\ge}.
    \end{equation*}
    This holds for every $\ge > 0$, so by taking the limit of both sides as $\ge \to 0$ we obtain
    \begin{equation}\label{elog}
    		e^{|\log (m(y) / m(x))|} \leq c_\ga^\prime.
    \end{equation}
    Without loss of generality, we can suppose that $m(x) \geq m(y)$.
    Then $|\log (m(y)/m(x))| = \log (m(x)/m(y))$, and \eqref{elog} implies that
    \begin{equation*}
      \frac{m(x)}{m(y)} \leq c_\ga^\prime,
    \end{equation*}
    which completes the proof.
  \end{proof}

  \begin{claim}
    Condition (A) is satisfied, with $C_\ga = D_\gm e^{3\ga e^{M\ga}}$.
  \end{claim}

  \begin{proof}
    As in the previous example, it suffices to show that for every 
    $B \in \mc{B}_\ga$ we have 
    \begin{equation}
      \label{localdoubling3}
      \begin{cases}
        e^{-\gfv(x)} \leq C_{\ga}^{\prime} e^{-\gfv(c_B)} , \quad \text{when  $x\in 2B$,} \\
        e^{-\gfv(x)} \geq C_{\ga}^{\prime\prime} e^{-\gfv(c_B)} , \quad \text{when $x\in B$.}
      \end{cases}
    \end{equation}	 
    This is implied (with $C_\ga^\prime = e^{\ga c_\ga^\prime}$ and $C_\ga^{\prime\prime} = e^{-2\ga c_\ga^\prime}$) by the estimate
    \begin{equation*}
      |\gfv(x) - \gfv(c_B)| \leq \gl \ga c_\ga^\prime \quad \forall x\in \gl B,
    \end{equation*}
    for all $\gl \geq 1$ and $x \in \gl B$, which we now show.
    If $x \in \gl B$, then we have
    \begin{equation*}
      |\gfv(x) - \gfv(c_B)| \leq \sup_{y \in \gl B} |\nabla \gfv(y)| d(x,c_B).
    \end{equation*}
    Since $B$ is $\ga$-admissible, for all $x,y \in \gl B$ Claim \ref{a12} yields
    \begin{equation*}
      d(x,c_B) \leq \gl r_B \leq \gl \ga m(c_B) \leq \gl \ga c_{\ga}^\prime m(y) \leq \gl \ga c_{\ga}^\prime |\nabla \gfv(y)|^{-1},
    \end{equation*}
    and so $|\gfv(x) - \gfv(c_B)| \leq \gl\ga c_\ga^\prime$.
    As in the previous example, we then have
    \begin{equation*}
      C_\ga = D_\gm \frac{C_\ga^\prime}{C_\ga^{\prime\prime}} = D_\gm e^{3\ga c_\ga^\prime}.
    \end{equation*}
    Using $c_\ga^\prime = e^{M\ga}$ (from Claim \ref{a12}) yields the result.
  \end{proof}

  For a concrete subexample, let $(X,d,\gm)$ be the Euclidean space $\RR^n$ with the usual Euclidean distance and Lebesgue measure, and let $\gfv \in \RR[x_1,\ldots,x_n]$ be a polynomial.
  Condition (B) is easily verified, although condition (H) may not hold when $n \geq 2$.
  Taking $\gfv(x) = \frac{n\log(2\gp)}{2} + \frac{1}{2}\sum_{i=1}^n x_i^2$, we again recover the setting of Gaussian harmonic analysis.
  However, in this case the constants $c_\ga^\prime$ and $C_\ga$ have significantly worse $\ga$-dependence than the constants we found in the previous example.
  This is because conditions (B) and (H) are less restrictive than assuming $\gfv$ is given in terms of a distance function.

\end{example}

\begin{rmk} 
    The utility of an admissibility function is eventually judged by its applicability to
    the local Hardy space theory. More precisely, one needs to obtain suitable `error estimates'
    in the spirit of \cite[Section 5]{pP13}. The only known example of such at the time of writing
    is the setting of $\RR^n$ with $\gfv (x) = \frac{n}{2} \log \pi + |x|^2$ and 
    $m(x) = \min (1 , |x|^{-1})$.    
\end{rmk}

\section{Local tent spaces: the reflexive range}\label{localts}

We now introduce the main topic of the paper --- the non-uniformly local tent spaces.
Let $\gfv$ and $m$ be given and satisfy (A) from Section \ref{basicsetting}.
Denote the resulting weighted measure by $\gg$.

\begin{dfn}
  Let $0 < p,q < \infty$ and $\ga > 0$. 
  The \emph{local tent space} $\mf{t}^{p,q}_\ga(\gg)$ is the set of all measurable functions $f$ defined on the \emph{admissible region}
  \begin{equation*}
    D = \{ (y,t)\in X\times (0,\infty) : t < m(y) \} 
  \end{equation*}
  such that the functional
  \begin{equation*}
    \mc{A}^\ga_q f(x) = \left( \iint_{\Gamma_\ga (x)} |f(y,t)|^q \,\frac{d\gg(y)}{\gg(B(y,t))} \frac{dt}{t} \right)^{1/q}
  \end{equation*}
  satisfies
  \begin{equation*}
    \| f \|_{\mf{t}^{p,q}_\ga(\gg)} := \| \mc{A}_q^\ga f \|_{L^p(\gg)} < \infty .
  \end{equation*}
  Here $\Gamma_\ga (x) = \{ (y,t)\in D : d(x,y) < \ga t \}$ is the \emph{admissible cone} of aperture $\ga$ at $x\in X$.
\end{dfn}

It is clear that $\norm{\cdot}_{\mf{t}^{p,q}_\ga(\gg)}$ is a norm on $\mf{t}^{p,q}(\gg)$ when $p,q \in [1,\infty)$, and a quasinorm when $p<1$ or $q<1$.
Following the argument of \cite[Proposition 3.4]{aA13} with doubling replaced by local doubling, we can show that $\mf{t}^{p,q}_\ga(\gg)$ is complete in this (quasi-)norm.

\begin{rmk}
The choice $\gfv = 0$ and $m = \infty$ recovers the setting of global tent spaces 
\cite{aA13}, whereas $\gfv = 0$ and $m = 1$ gives the setting of uniformly local tent spaces
by Carbonaro, McIntosh and Morris \cite{CMM13}.
\end{rmk}

For $1 < p,q < \infty$, the properties of $\mf{t}^{p,q}_\ga(\gg)$ can be studied, as in \cite{HTV91}, by embedding the space into an $L^p$-space of $L^q$-valued functions.
More precisely, let us write $L^q(D)$ for the space of $q$-integrable functions on $D$ with respect to the measure $\frac{d\gg (y)\,dt}{t\gg (B(y,t))}$, so that
\begin{equation*}
  \injmap{J_\ga}{\mf{t}^{p,q}_\ga(\gg)}{L^p (\gg ; L^q(D))} , \quad J_\ga f(x) = \mb{1}_{\gG_\ga (x)}f
\end{equation*}
defines an isometry.
We will show that $J_\ga$ embeds $\mf{t}^{p,q}_\ga(\gg)$ as a complemented subspace of $L^p(\gg;L^q(D))$, with
\begin{equation*}
  N_\ga U(x;y,t) = \mb{1}_{B(y,\ga t)}(x) \fint_{B(y,\ga t)} U(z;y,t) \,d\gg (z) , \quad
  (U\in L^p(\gg; L^q(D)), \; x\in X , \; (y,t)\in D)
\end{equation*}
defining a bounded projection of $L^p(\gg;L^q(D))$ onto the image of $\mf{t}^{p,q}_\ga(\gg)$.

To see that $N_\ga$ is bounded, we first observe that
\begin{align*}
  |N_\ga U(x;y,t)| &\leq \mb{1}_{B(y,\ga t)}(x) \fint_{B(y,\ga t)} |U(z;y,t)| \,d\gg (z) \\
  &\leq \sup_{\substack{B\ni x \\ B\in\mc{B}_\ga}} \fint_B |U(z;y,t)| \,d\gg (z) \\
  &= \mc{M}_\ga U (x;y,t) ,
\end{align*}
where $\mc{M}_\ga$ is the $L^q(\gS)$-valued $\ga$-local maximal function from Appendix \ref{lmf}, with $\Sigma = (D,\frac{d\gg(y) \, dt}{t\gg(B(y,t))})$.
Consequently,
\begin{equation*}
  \| N_\ga U \|_{L^p(\gg ; L^q(D))} \leq \| \mc{M}_\ga U \|_{L^p(\gg ; L^q(D))} \lesssim_{p,q}
  C_{\ga , c_X} \| U \|_{L^p(\gg ; L^q(D))},
\end{equation*}
(see Appendix \ref{lmf}).

An immediate consequence of this vector-valued approach is the following theorem, detailing the behaviour of the local tent spaces in the reflexive range.

\begin{thm}
  \label{refl}
  Let $1 < p,q < \infty$. We have
  \begin{itemize}
  \item (change of aperture)
    $\| f \|_{\mf{t}^{p,q}_\ga (\gg)} \eqsim_{p,q,\ga,\gb} \| f \|_{\mf{t}_\gb^{p,q} (\gg)}$ for 
    $0 < \gb < \ga < \infty$,
  \item (duality) $\mf{t}^{p,q}_\ga(\gg)^* = \mf{t}_\ga^{p^\prime,q^\prime}(\gg)$, realised by the duality pairing
    \begin{equation*}
      \langle f , g \rangle = \iint_D f(y,t) \overline{g(y,t)} \, d\gg(y) \frac{dt}{t} ,
    \end{equation*}
  \item (complex interpolation) $[\mf{t}_\ga^{p_0,q_0}(\gg) , \mf{t}_\ga^{p_1,q_1}(\gg)]_\theta = \mf{t}_\ga^{p,q}(\gg)$ when $1 < p_0\leq p_1 < \infty$ and $1 < q_0 \leq q_1 < \infty$, with $1/p = (1-\theta)/p_0 + \theta /p_1$, $1/q = (1-\theta)/q_0 + \theta/q_1$.
  \end{itemize}
  \begin{proof}
    For our claim on change of aperture, we follow \cite{HTV91} and begin by noting that for suitable $f$ we have
    \begin{equation*}
      N_\ga J_\gb f(x;y,t) = \frac{\gg (B(y,\gb t))}{\gg (B(y,\ga t))} J_\ga f(x;y,t) .
    \end{equation*}
    Then
    \begin{align*}
      \| f \|_{\mf{t}^{p,q}_\ga (\gg)} = \| J_\ga f \|_{L^p(\gg ; L^q(D))} 
      &= \frac{\gg (B(y,\ga t))}{\gg (B(y,\gb t))} \| N_\ga J_\gb f \|_{L^p(\gg ; L^q(D))} \\
      &\lesssim_{p,q} C_{\gb , \ga / \gb} C_{\ga , c_X} \| J_\gb f \|_{L^p(\gg ; L^q(D))}
      = C_{\gb , \ga / \gb} C_{\ga , c_X} \| f \|_{\mf{t}^{p,q}_\gb (\gg)} ,
    \end{align*}
    where the constant are from Remark \ref{doublingconstant}.
    
    Now $\mf{t}^{p,q}_\ga(\gg)$ is embedded in $L^p(\gg ; L^q(D))$ as the range of the projection $N_\ga$, whose dual is isomorphic to the range of $N_\ga^*$ on $L^p(\gg ; L^q(D))^* = L^{p'}(\gg ; L^{q^\prime}(D))$, which, in turn, is isometrically isomorphic to $\mf{t}^{p',q^\prime}_\ga(\gg)$ (because $N_\ga^* = N_\ga$).
    The duality is realised as
    \begin{align*}
      \langle f , g \rangle = \langle J_\ga f , J_\ga g \rangle
      = \int_X \langle \mb{1}_{\gG_\ga (x)}f , \mb{1}_{\gG_\ga (x)}g \rangle \, d\gg (x)
      &= \int_X \iint_{\gG_\ga (x)} f(y,t) \overline{g(y,t)} \, 
      \frac{d\gg (y)}{\gg(B(y,t))}\frac{dt}{t} d\gg(x) \\
      &= \iint_D f(y,t) \overline{g(y,t)} \, 
      d \gg(y) \frac{dt}{t}.
    \end{align*}
    
    For $1 < p_0 \leq p_1 < \infty$ and $1 < q_0 \leq q_1 < \infty$ the interpolation of tent spaces follows, by the standard result on interpolation of complemented subspaces \cite[Section 1.17]{hT78}, from the fact that
    \begin{equation*}
      [L^{p_0}(\gg ; L^{q_0}(D)) , L^{p_1}(\gg ; L^{q_1}(D))]_\theta = L^{p}(\gg ; L^q(D)).
    \end{equation*}
  \end{proof}
\end{thm}

\begin{rmk}\label{postreflexive}
    The dependence on $\ga > 1$ in the aperture change constant $C_{1,\ga} C_{\ga,c_X}$ (between
    $\mf{t}^{p,q}_\ga(\gg)$ and $\mf{t}^{p,q}_1(\gg)$) is not optimal in general.
    For instance, on $(\RR^n , dx)$, the optimal dependence is $\alpha^{n/\min (p,2)}$ (see \cite{pA11}), while $C_{1,\ga}C_{\ga,c_X} \eqsim \alpha^n$. 
    Note however, that on $(\RR^n, \gg)$ we have $C_{1,\ga}C_{\ga,c_X} \lesssim e^{c\ga^2}$ for some constant $c$.
    We return to this in Section \ref{nncurv}.
\end{rmk}

The change of aperture and interpolation results extend to $1 \leq p,q < \infty$ by a convex reduction due to Bernal (\cite{aB92}, see also \cite{aA13}).

\begin{cor}
\label{bernal}
  Let $1 \leq q < \infty$. We have
  \begin{itemize}
  \item (change of aperture)
    $\| f \|_{\mf{t}^{1,q}_\ga (\gg)} \eqsim_{q,\ga,\gb} \| f \|_{\mf{t}_\gb^{1,q} (\gg)}$ for 
    $0 < \gb < \ga < \infty$,
  \item (complex interpolation) $[\mf{t}_\ga^{p_0,q_0}(\gg) , \mf{t}_\ga^{p_1,q_1}(\gg)]_\theta = \mf{t}_\ga^{p,q}(\gg)$ when $1 \leq p_0\leq p_1 < \infty$ and $1 < q_0 \leq q_1 < \infty$, with $1/p = (1-\theta)/p_0 + \theta /p_1$, $1/q = (1-\theta)/q_0 + \theta/q_1$.
  \end{itemize}
\end{cor}

\section{Endpoints: $\mf{t}^{1,q}$ and $\mf{t}^{\infty,q}$}\label{nncurv}

In this section, under the assumption that the space $X$ is complete, we study the endpoints of the local tent space scale: the spaces $\mf{t}^{1,q}_\ga(\gg)$ and $\mf{t}^{\infty,q}_\ga(\gg)$ (with $1 \leq q < \infty$).
In particular, employing Corollary \ref{bernal} we prove following the argument in \cite{mK11} that elements of $\mf{t}^{1,q}_\ga(\gg)$ can be decomposed into `atoms'. From this we deduce duality, interpolation, and (quantified) change of aperture results for the full local tent space scale $\mf{t}^{p,q}_\ga(\gg)$ ($1 \leq p \leq \infty$, $1 \leq q < \infty$).\footnote{We do not consider $q = \infty$. As in \cite{CMS85}, this requires additional continuity and convergence assumptions.} We write $\mf{t}^{1,q} := \mf{t}^{1,q}_1$ for notational simplicity.

\subsection{Atomic decomposition}

 Fix $(X,d,\gm)$, $\gfv$, and $m$ as in the previous section.
 The \emph{admissible tent} $T(O)$ over an open set $O \subset X$ is given by
\begin{equation*}
  T(O) := D \sm \gG(O^c),
\end{equation*}
where $\gG(O^c) := \cup_{x \in O^c} \gG(x)$.

\begin{dfn}
  Fix $\ga > 0$ and $q \geq 1$.
  A function $a$ on $D$ is called an \emph{$\ga$-$\mf{t}^{1,q}$-atom} (or more succinctly, a $\ga$-atom) if there exists an $\ga$-admissible ball $B\in\mc{B}_\ga$ such that $\supp a \subset T(B)$ and
  \begin{equation*}
    \iint_{T(B)} |a(y,t)|^q \, d\gg(y) \frac{dt}{t} \leq \frac{1}{\gamma (B)^{q-1}}.
  \end{equation*}
\end{dfn}

 Observe that for such a function $a$,
\begin{equation*}
  \| a \|_{\mf{t}^{1,q}(\gg)} = \int_B \mc{A}_qa(x) \, d\gg (x) 
  \leq \gg(B)^{\frac{q-1}{q}} \left( \int_B \mc{A}_qa(x)^q \, d\gg (x) \right)^{1/q} \lesssim 1 .
\end{equation*}
Furthermore, if $(a_k)_{k \in \NN}$ is a sequence of $\ga$-$\mf{t}^{1,q}$-atoms for some $\ga > 0$, then the series $f = \sum_k \gl_k a_k$ converges in $\mf{t}^{1,q}(\gg)$ when $\sum_k |\gl_k| < \infty$.
 The \emph{atomic tent space} $\mf{t}^{1,q}_{\text{at}}(\gg)$ consisting of such functions $f$ becomes
 a Banach space when normed by
\begin{equation*}
  \| f \|_{\mf{t}^{1,q}_{\text{at}}(\gg)} = \inf \Big\{ \sum_k |\gl_k| : f = \sum_k \gl_k a_k \Big\} .
\end{equation*}

 \begin{lem}\label{balls}
  Suppose that $E\subset X$ is a bounded open set. Then there exists a countable sequence of disjoint admissible balls
  $B^j\subset E$ such that
  \begin{equation*}
    T(E) \subset \bigcup_{j\geq 1} T(5B^j) .
  \end{equation*}
\end{lem}
\begin{proof}
  Let $\gd_1 = \sup \{ r_B : B\subset E \text{ admissible} \}$ and begin by choosing an
  admissible ball
  $B^1\subset E$ with radius $r_1 > \gd_1 / 2$. Proceeding inductively we put
  \begin{equation*}
    \gd_{k+1} = \sup \{ r_B : B\subset E \text{ admissible} , B\cap B^j = \varnothing , j=1,\ldots , k \}
  \end{equation*}
  and choose (if possible) an admissible ball $B^{k+1}\subset E$ with radius $r_{k+1} > \gd_{k+1}/2$
  disjoint from $B^1,\ldots , B^k$. Given a $(y,t)\in T(E)$ we show that 
  $B(y,t)\subset 5B^j$ for some $j$.
  It is possible to pick the first index $j$ for which $B(y,t)\cap B^j$ is nonempty. Indeed, if
  on the contrary $B(y,t)$ was disjoint from every $B^j$, then, $B(y,t)$ being admissible and contained in $E$, we would
  have $t\geq \gd_j$ for all $j$ which under the assumption that $(X,d)$ is geometrically doubling contradicts the
  boundedness of $E$.
  By construction, we then have $t\leq \gd_j \leq 2r_j$ and so $B(y,t)\subset 5B^j$, as required.
\end{proof}

\begin{rmk}
  The above lemma is a stronger version of a `local Vitali covering lemma', 
  which is otherwise identical but claims only that
  $E\subset \bigcup_{j\geq 1} 5B^j$ without reference to tents (see also Remark \ref{locvitali} in the Appendix).
\end{rmk}

The following lemma regarding pointwise estimates for $\mc{A}$-functionals, which appears implicitly in \cite[Theorem 4']{CMS85}, lies at the heart of our proof of the atomic
decomposition.
This is the only point at which we seem to need completeness; we suspect that this assumption can be removed or at least weakened.
\begin{lem}
\label{pointwise2}
Suppose $X$ is complete, let $q \geq 1$ and
let $f$ be a measurable function in $D$. Let $\lambda > 0$ and write $E = \{ x\in X : \mc{A}_q^3 f(x) > \lambda \}$. Then
$\mc{A}_q(f\mb{1}_{D\sm T(E)})(x) \leq \lambda$ for all $x\in X$.
\begin{proof}
  If $x\not\in E$, then $\mc{A}_q(f\mb{1}_{D\sm T(E)})(x) \leq \mc{A}_q^3f(x) \leq \lambda$. 
  
  If $x\in E$, then by completeness of $X$ we can choose a point
  $x_0\in X\sm E$ such that $d(x,x_0) = d(x,X\sm E)$. We show that $\Gamma (x) \sm T(E) \subset \Gamma^3(x_0)$:
  Let $(y,t) \in \Gamma (x) \sm T(E)$ so that $d(x,y)<t$ and  $B(y,t)\not\subset E$. Now $B(y,t)\subset B(x,2t)$,
  which means that $B(x,2t)\not\subset E$ and so $x_0\in B(x,2t)$. Moreover $B(x,2t)\subset B(y,3t)$ so that
  $(y,t)\in \Gamma^3(x_0)$. Therefore $\mc{A}_q(f\mb{1}_{D\sm T(E)})(x) \leq \mc{A}_q^3f(x_0) \leq \lambda$.
\end{proof}
\end{lem}

\begin{thm}\label{atomic}
  Suppose $X$ is complete, and let $q \geq 1$.
  For every $f \in \mf{t}^{1,q}(\gg)$, there exist $5$-$\mf{t}^{1,q}$-atoms $a_k$ and scalars $\gl_k$ such that
  \begin{equation}\label{ad}
    f = \sum_k \gl_k a_k,
  \end{equation}
  with
  \begin{equation*}
    \sum_k |\gl_k| \simeq \norm{f}_{\mf{t}^{1,q}(\gg)} .
  \end{equation*}
\end{thm}

We call the series \eqref{ad} an \emph{atomic decomposition} of $f$.

\begin{proof}
  We first derive atomic decompositions for the dense class of boundedly-supported functions in $\mf{t}^{1,q}(\gg)$, and then argue by completeness of $\mf{t}^{1,q}_{\text{at}}(\gg)$.
  Given a function $f$ in $\mf{t}^{1,q}(\gamma)$ with bounded support, we consider the bounded open sets
  \begin{equation*}
    E_k = \{ x\in X : \mc{A}_q^3 f(x) > 2^k \} 
  \end{equation*}
  for each integer $k$.
  Applying Lemma \ref{balls} to these sets provides us with 
  disjoint balls $B_k^j\subset E_k$ such that
  \begin{equation*}
    T(E_k) \subset \bigcup_{j\geq 1} T(5B_k^j) .
  \end{equation*}
  In addition, we take a collection of functions $\chi_k^j$ (cf. \cite[Theorem 11]{mK11}) satisfying
  \begin{equation*}    
    0\leq \chi_k^j \leq 1, \quad 
    \sum_{j\geq 1} \chi_k^j = 1 \text{ on } T(E_k), \quad \text{and} \quad 
    \supp \chi_k^j \subset T(5B_k^j).
  \end{equation*}
  Writing $A_k := T(E_k)\setminus T(E_{k+1})$, we can decompose $f$ as
  \begin{equation*}
    f = \sum_{k\in\ZZ} \mb{1}_{A_k} f
    = \sum_{k\in\ZZ} \sum_{j\geq 1} \chi_k^j \mb{1}_{A_k} f
    = \sum_{k\in\ZZ} \sum_{j\geq 1} \lambda_k^j a_k^j ,
  \end{equation*}
  where
  \begin{equation*}
    \lambda_k^j = \gamma (5B_k^j)^{1/{q^\prime}} \left( \int_{5B_k^j} 
      \mc{A}_q(f\mb{1}_{A_k})(x)^q \, d\gamma(x) \right)^{1/q} .
  \end{equation*}
  Observe that $a_k^j = \chi_k^j \mb{1}_{A_k} f / \lambda_k^j$ is a $5$-atom supported in $T(5B_k^j)$.
  
  What remains is to control the sum of the scalars $\lambda_k^j$.
  By Lemma \ref{pointwise2}, we have
  \begin{equation*}
  \mc{A}_q\left(f\mb{1}_{A_k}\right)(x) \leq \mc{A}_q\left(f\mb{1}_{D\sm T(E_{k+1})}\right)(x) \leq 2^{k+1}
  \end{equation*}
  for all $x\in X$, and so
  \begin{equation*}
    \gl_k^j \leq \gg (5B_k^j) 2^{k+1} .
  \end{equation*}
  Consequently,
  \begin{equation*}
  \sum_{k\in\ZZ} \sum_{j\geq 1} \lambda_k^j 
  \leq \sum_{k\in\ZZ} 2^{k+1} \sum_{j\geq 1} \gamma (5B_k^j)  
  \lesssim \sum_{k\in\ZZ} 2^{k+1}\gamma (E_k)
  \lesssim \| \mc{A}_q^3 f \|_{L^1(\gamma)}
  \lesssim \| f \|_{\mf{t}^{1,q}(\gamma)} ,
  \end{equation*}
where the last step follows by Corollary \ref{bernal}.
  
  We have thus shown that $\| f \|_{\mf{t}^{1,q}_{\text{at}}(\gg)} \eqsim \| f \|_{\mf{t}^{1,q}(\gg)}$ for boundedly supported $f$ in $\mf{t}^{1,q}(\gg)$.
  Since the class of such functions is dense in $\mf{t}^{1,q}(\gg)$, the completeness of $\mf{t}^{1,q}_{\text{at}}(\gg)$ guarantees that every $f\in\mf{t}^{1,q}(\gg)$ has an atomic decomposition.
\end{proof}

\begin{rmk}\label{mvnpatoms}
  Maas, van Neerven and Portal established the above result in the setting of Gaussian $\RR^n$ by a different method, which relies on Gaussian Whitney decompositions \cite[Theorem 3.4]{MvNP12}.
  In addition, they showed that decompositions into $\ga$-atoms exist for every $\ga > 1$ \cite[Lemma 3.6]{MvNP12}.
  Such a result may not hold in this level of generality due to the lack of geometric information.
\end{rmk}

\subsection{Duality, interpolation and change of aperture}

We present three corollaries of the atomic decomposition theorem, which holds when $X$ is complete.

The dual of $\mf{t}^{1,q}(\gg)$ can be identified with the space $\mf{t}^{\infty,q^\prime} (\gg)$, consisting of those functions $g$ on $D$ for which
\begin{equation*}
  \| g \|_{\mf{t}^{\infty,{q^\prime}} (\gg)} = \sup_{B\in\mc{B}_5} 
  \left( \frac{1}{\gg (B)} \iint_{T(B)} |g(y,t)|^{q^\prime} \, d\gg(y) \frac{dt}{t} \right)^{1/{q^\prime}} 
\end{equation*}
is finite.
Note that we take a supremum over $5$-admissible balls, reflecting the fact that we have atomic decompositions of elements of $\mf{t}^{1,q}(\gg)$ into $5$-atoms.
For the reader's convenience, we present the standard proof, following 
\cite[Theorem 1 (b)]{CMS85}.

\begin{cor}\label{dualitythm}
   Suppose $X$ is complete, and let $q \geq 1$.
   Then the pairing
   \begin{equation}
     \label{duality}
     \langle f , g \rangle = \iint_D f(y,t) \overline{g(y,t)} \, d\gg(y) \frac{dt}{t} , 
     \quad f\in \mf{t}^{1,q}(\gamma) , \quad g\in \mf{t}^{\infty,{q^\prime}} (\gamma) ,
   \end{equation}
   realises $\mf{t}^{\infty,q^\prime} (\gamma)$ as the dual of $\mf{t}^{1,q}(\gamma)$.
\end{cor}
   \begin{proof}
     To see that \eqref{duality} defines a bounded linear functional on $\mf{t}^{1,q}(\gamma)$  for every $g\in\mf{t}^{\infty,q^\prime}(\gamma)$, it suffices (by Theorem \ref{atomic}) to test the pairing on atoms.
    For any atom $a$ associated with a ball $B\in\mc{B}_5$ we have
     \begin{equation*}
      |\langle a , g \rangle | \leq \iint_{T(B)} |a \overline{g}| \, d\gg \frac{dt}{t}
      \leq \left( \iint_{T(B)} |a|^q \, d\gg \frac{dt}{t} \right)^{1/q}
      \left( \iint_{T(B)} |g|^{q^\prime} \, d\gg \frac{dt}{t} \right)^{1/{q^\prime}}
      \leq \| g \|_{\mf{t}^{\infty,q^\prime} (\gamma)} .
    \end{equation*}

    To show that every functional $\Lambda \in \mf{t}^{1,q^\prime}(\gamma)^*$ arises in this way, we first note that each $f\in L^q(T(B)),$\footnote{We equip the space $T(B)$ with the product measure $d\gg(y) dt/t$.} with $B\in\mc{B}_5$, satisfies
    \begin{equation*}
      \| f \|_{\mf{t}^{1,q}(\gamma)} \leq \gamma (B)^{1/{q^\prime}} \| f \|_{L^q(T(B))} .
    \end{equation*}
    Hence $\Lambda$ restricts to a bounded linear functional on $L^q(T(B))$, and is thus given by
    \begin{equation*}
      \Lambda f = \iint_{T(B)} f \overline{g_B} \, d\gg \frac{dt}{t} , \quad f\in L^q(T(B)),
    \end{equation*}
    for some $g_B \in L^{q^\prime}(T(B))$, with the estimate
    \begin{equation*}
      \| g_B \|_{L^{q^\prime}(T(B))} \leq \gg (B)^{1/{q^\prime}} \| \Lambda \|_{\mf{t}^{1,q^\prime}(\gg)^*} .
    \end{equation*}
    A single function $g$ on $D$ can then be obtained from the family $(g_B)_{B\in\mc{B}_5}$ in a well-defined manner, since for any two balls $B,B^\prime \in \mc{B}_5$, the functions $g_B$ and $g_{B'}$ agree on $T(B)\cap T(B')$. 
    It remains to be checked that $\| g \|_{\mf{t}^{\infty,q^\prime}(\gg)} \simeq \| \Lambda \|_{\mf{t}^1(\gg)^*}$. 
    On the one hand,
    for any $B\in\mc{B}_5$ we have
    \begin{equation*}
      \left( \iint_{T(B)} | g |^{q^\prime} \, d\gg \frac{dt}{t} \right)^{1/{q^\prime}}
      = \| g_B \|_{L^{q^\prime}(T(B))} \leq \gamma (B)^{1/{q^\prime}} \| \Lambda \|_{\mf{t}^{1,q^\prime}(\gg)^*} .
    \end{equation*}
    On the other hand, due to Theorem \ref{atomic}, $\| \Lambda \|_{\mf{t}^{1,q}(\gg)^*}$ is achieved (up to a constant) 
    by testing against all atoms, and so the proof is completed after checking that
    \begin{align*}
      |\Lambda a| &\leq \iint_{T(B)} |g \overline{a}| \, d\gg \frac{dt}{t} \\ 
      &\leq \left( \iint_{T(B)} |g|^{q^\prime} \, d\gg \frac{dt}{t} \right)^{1/q^\prime}
      \left( \iint_{T(B)} |a|^{q} \, d\gg \frac{dt}{t} \right)^{1/q} \\
      &\leq \gg(B)^{1/q^\prime} \norm{g}_{\mf{t}^{\infty,q^\prime}(X,\gg)} \gg(B)^{-1/q^\prime} \\
      &= \norm{g}_{t^{\infty,q^\prime}(\gg)} .
    \end{align*}
   \end{proof}

   \begin{cor}\label{interpolation}
     Suppose $X$ is complete.
     For $1\leq p_0 \leq p_1 \leq \infty$ (excluding the case $p_0 = p_1 = \infty$) and $1 \leq q_0 \leq q_1 < \infty$, we have
     $[\mf{t}^{p_0,q_0}(\gamma ) , \mf{t}^{p_1,q_1}(\gamma)]_\theta = \mf{t}^{p,q}(\gamma)$, when 
     $1/p = (1-\theta)/p_0 + \theta /p_1$, $1/q = (1-\theta)/q_0 + \theta / q_1$, and $0\leq\theta\leq 1$.
   \end{cor}
   \begin{proof}
     This follows directly from Theorem \ref{refl} and Corollary \ref{dualitythm}, by convex reduction and reiteration (see Remark \ref{postreflexive}).
   \end{proof}

   \begin{cor}\label{coa}
     Let $q \geq 1$.
     For all $1\leq p \leq q$ and $\ga \geq 1$ we have
     \begin{equation*}
       \| f \|_{\mf{t}^{p,q}_\ga (\gg)} \lesssim C_{1,\ga}^{1/q} C_{5,\ga}^{1/p - 1/q}
       \| f \|_{\mf{t}^{p,q} (\gg)} .
  \end{equation*}
   \end{cor}
   \begin{proof}
     In order to argue by interpolation, consider first the case $p=q$:
    \begin{align*}
      \| f \|_{\mf{t}^{q,q}_\ga (\gg)}^q &= \int_X \iint_{\Gamma^\ga (x)} |f(y,t)|^q
      \, \frac{d\gg(y)}{\gg(B(y,t))} \frac{dt}{t} \, d\gg(x) \\
      &= \int_X \int_0^\infty |f(y,t)|^q \mb{1}_{(0,m(y))}(t) \frac{\gg (B(y,\ga t))}{\gg (B(y,t))}
      \,\frac{dt}{t} \,d\gg (y)\\
      &\leq C_{1,\ga} \iint_D |f(y,t)|^q \, d\gg(y) \frac{dt}{t} \\
      &= C_{1,\ga} \| f \|_{\mf{t}^{q,q} (\gg)}^q .
    \end{align*}
     For $p=1$ we make use of the atomic decomposition.
     If $a$ is a $5$-atom associated with $B\in\mc{B}_5$,
     then, since $\Gamma_\ga (x) \cap T(B)$ is non-empty exactly when $x\in\ga B$, we have
    \begin{align*}
      \| a \|_{\mf{t}^{1,q}_\ga (\gg)} &\leq \gg (\ga B)^{1/{q^\prime}} \| a \|_{\mf{t}^{q,q}_\ga (\gg)} \\
      &\leq C_{1,\ga}^{1/q} \gg (\ga B)^{1/q^\prime} \| a \|_{\mf{t}^{q,q} (\gg)} \\
      &\leq C_{1,\ga}^{1/q} \left( \frac{\gg (\ga B)}{\gg (B)} \right)^{1/q^\prime} \\
      &\leq C_{1,\ga}^{1/q} C_{5,\ga}^{1-1/q} .
    \end{align*}
    Thus $\| f \|_{\mf{t}^{1,q}_\ga (\gg)} \leq C_{1,\ga}^{1/q} C_{5,\ga}^{1-1/q} \| f \|_{\mf{t}^{1,q} (\gg)}$ for all $f\in\mf{t}^{1,q} (\gg)$, and the result then follows by interpolation.
   \end{proof}

\begin{rmk}
  Note that on $(\RR^n, dx)$ this gives the optimal dependence on $\ga$ for $1\leq p \leq 2$, which we could not obtain from the vector-valued approach, since $C_{1,\ga}^{1/2} C_{5,\ga}^{1/p - 1/2} = \ga^{n/p}$ (see Remark \ref{postreflexive}).
  On Gaussian $\RR^n$ this merely extends the aperture change to $\mf{t}^1(\gg)$ with the constant $e^{c\ga^2}$, the improvement from interpolation being immaterial.
\end{rmk}

\begin{appendix}

  \section{Local maximal functions}\label{lmf}

  Here we present a brief justification of the boundedness of the maximal functions used above and in Appendix \ref{cone}.
  We use dyadic methods, particularly the existence of finitely many `adjacent' dyadic systems, combined with some methods from Martingale theory.
  At the end of this section we indicate another approach, which is more elementary but does not adapt well to vector-valued contexts.

  By a \emph{dyadic system} on a measure space $(X,\gg)$ we mean a countable collection $\mc{D} = \{\mc{D}_k\}_{k \in \ZZ}$, where each $\mc{D}_k$ is a partition of $X$ into measurable sets of finite nonzero measure, such that the containment relations
  \begin{equation*}
    Q\in\mc{D}_k , \quad R\in\mc{D}_l , \quad l\geq k \quad \Longrightarrow
    \quad R\subset Q \quad \text{or} \quad Q\cap R = \emptyset
  \end{equation*}
  hold.
  The elements of $\mc{D}_k$ are called \emph{dyadic cubes (of generation $k$)}.

  Associated to each dyadic system $\mc{D}$ is a \emph{dyadic maximal function}, defined by
  \begin{equation*}
    M_\mc{D}u(x) = \sup_{\substack{Q\ni x \\ Q\in\mc{D}}} \fint_Q |u| \,d\gg
  \end{equation*}
  for all $u \in L_\text{loc}^1(\gg)$.
  Since $M_\mc{D}$ coincides with the martingale maximal operator for the (increasing) filtration $(\mc{F}_k)_{k\in\ZZ}$ 
  when each $\mc{F}_k$ is the $\sigma$-algebra generated by $\mc{D}_k$, it follows that
  $M_\mc{D}$ satisfies a weak type-(1,1) inequality
  \begin{equation}\label{wt11}
    \gg (\{ x\in X : M_\mc{D}u(x) > \lambda \} ) \leq \frac{1}{\lambda} \| u \|_{L^1(\gg)}
  \end{equation}
  for all $\lambda > 0$ (see for instance \cite[Theorem 14.6]{dW91} or \cite[Chapter IV, Section 1]{emS70}). 

  Now suppose that $(X,d)$ is a geometrically doubling metric space.
  Hyt\"onen and Kairema showed in \cite{HK12} (see also \cite{tM03}) the existence of a finite collection of \emph{adjacent dyadic systems}.

  \begin{thm}
    \label{adjacent}
    There exists a finite collection $\{\mc{D}_i\}_{i=1}^{N}$ of dyadic systems on $X$, with $N$ bounded by a constant depending only on the geometric doubling constant of $(X,d)$, such that every open ball $B\subset X$ is contained in a dyadic cube $Q_B$ from one of the dyadic systems, with $\diam(Q_B) \leq c_X \diam(B)$.
  \end{thm}

  Now let $(X,d,\gm)$, $\gg$, and $m$ be as in Section \ref{basicsetting}, and let $\ga > 0$.
  Combining the theorem above with the weak type-(1,1) estimate for the dyadic maximal function yields a corresponding weak type-(1,1) estimate for the $\ga$-local maximal operator $M_\ga$.

  Indeed, for each $\ga$-admissible ball $B \in \mc{B}_\ga$ we have that $B \subset Q_B$ for some dyadic cube $Q_B$ that satisfies $Q_B \subset c_X B$, and so
  \begin{align*}
    \mb{1}_B(x) \fint_B |u| \, d\gg &\leq \mb{1}_{Q_B}(x) \frac{\gg(Q_B)}{\gg(B)} \fint_{Q_B} |u| \, d\gg \\
    &\leq \mb{1}_{Q_B}(x) \frac{\gg(c_X B)}{\gg(B)} \fint_{Q_B} |u| \, d\gg \\
    &\leq \mb{1}_{Q_B}(x) C_{\ga,c_X} \fint_{Q_B} |u| \; d\gg.
  \end{align*}
  Here $C_{\ga,c_X}$ is the doubling constant from Remark \ref{doublingconstant}.
  Summing over finitely many dyadic systems, we find that
  \begin{equation*}
    M_\ga u(x) \leq C_{\ga,c_X} \sum_{\mc{D}} M_\mc{D} u(x),
  \end{equation*}
  and using the estimate \eqref{wt11} yields
  \begin{equation*}
    \gg(\{x \in X : M_\ga u(x) > \gl \}) \lesssim C_{\ga,c_X} \norm{u}_{L^1(\gg)}
  \end{equation*}
  for all $\gl > 0$.

  Similarly, given a $\gs$-finite measure space $\Sigma$, we can consider the $\ga$-local 
  lattice maximal operator $\mc{M}_\ga$, given by
  \begin{equation*}
    \mc{M}_\ga U(x,s) = \sup_{\substack{B\in\mc{B}_\ga \\ B\ni x}} \fint_B |U(z,s)| \,d\gg (z)
  \end{equation*}
  for $U \in L_\text{loc}^1(\gg;L^q(\gS))$ with $q \in (1,\infty)$ (see \cite{RdF86} for a general overview).
  Again, this is controlled pointwise by a finite sum of its dyadic counterparts, that is,
  \begin{equation}\label{dyadiccontrol}
    \mc{M}_\ga U(x,s) \leq C_{\ga , c_X} \sum_\mc{D} \mc{M}_\mc{D}U(x,s)
  \end{equation}
  for some finite collection of dyadic systems $\mc{D}$.
  The dyadic lattice maximal operators $\mc{M}_\mc{D}$ are again amenable to Martingale theory. Indeed, 
  according to the martingale version of Fefferman--Stein inequality (see \cite[Subsection 3.1]{MT00}) 
  we have for $1 < p < \infty$ that
  \begin{equation*}
    \| \mc{M}_\mc{D} U \|_{L^p(\gg ; L^q(\Sigma))} \lesssim_{p,q} \| U \|_{L^p(\gg ; L^q(\Sigma))},
  \end{equation*}
  and consequently
  \begin{equation*}
    \| \mc{M}_\ga U \|_{L^p(\gg ; L^q(\Sigma))} \lesssim_{p,q} C_{\ga , c_X} \| U \|_{L^p(\gg ; L^q(\Sigma))} .
  \end{equation*} 
  Although the explicit statement in \cite{MT00} concerns the case of sequences, i.e. the case $\Sigma = \NN$,
  it immediately extends to more general measure spaces $\Sigma$ by means of \emph{lattice finite representability}:
  $L^q(\Sigma)$ is lattice finitely representable in $\ell^q$ in the sense that for every
  finite dimensional sublattice $E$ of $L^q(\Sigma)$ and every $\varepsilon > 0$ there exists a
  sublattice $F$ of $\ell^q$ and a lattice isomorphism $\Phi : E\to F$ for which 
  $\| \Phi \| \| \Phi^{-1} \| \leq 1 + \varepsilon$ (see for instance \cite{GPLMR01} and the references therein). 
  For boundedness of $\mc{M}_\mc{D}$ it suffices to consider
  simple functions $U : X \to L^q(\Sigma)$ and the boundedness is therefore transferable in lattice finite
  representability.

  \begin{rmk}\label{locvitali}
    Martingale theory can be avoided by analysing $M_\ga$ by means of a `local Vitali covering lemma', analogous to the usual analysis of the (global) maximal operator through the usual Vitali covering lemma.
    One can then prove the duality of $\mf{t}^{p,q}_\ga$ and $\mf{t}^{p^\prime,q^\prime}_\ga$ for $1 < p,q < \infty$, and recover the boundedness of the projections $N_\ga$ by realising them as the adjoints of the (bounded) inclusions from $\mf{t}^{p,q}_\ga$ into the appropriate $L^q$-valued $L^p$-space.
    This is the method of Bernal \cite{aB92}, used by the first author for global tent spaces in \cite{aA13}.
    In this way we also avoid the use of the $L^q(\gS)$-valued maximal function $\mc{M}_\ga$, but we do not achieve the potential generality of the above method.
  \end{rmk}

\section{Cone covering lemma for non-negatively curved Riemannian manifolds}\label{cone}

In this section we prove a stronger version of Lemma \ref{pointwise2} that will be useful for the theory
of vector-valued tent spaces. This is based on a `cone covering lemma', the Euclidean version of which appears in
\cite[Lemma 10]{mK11}.

\subsection{Review of non-negatively curved spaces}

Recall that a complete length space $(X,d)$ has \emph{non-negative curvature} if and only if for every point $x \in X$ and for every pair of geodesics $\gr_1, \gr_2$ with $\gr_1(0) = \gr_2(0) = x$, the comparison angle\footnote{This is the corresponding angle of a Euclidean triangle with sidelengths $d(x,\gr_1(t))$, $d(x,\gr_2(t))$, and $d(\gr_1(t),\gr_2(t))$.}
\begin{equation*}
  \angle \gr_1(t) x \gr_2 (t) := \cos^{-1} \left( \frac{d(x,\gr_1(t))^2 + d(x,\gr_2(t))^2 - d(\gr_1(t),\gr_2(t))}{2d(x,\gr_1(t))d(x,\gr_2(t))} \right)
\end{equation*}
is nonincreasing in $t$.
Actually, this is a combination of the usual (local) definition of non-negative curvature and the conclusion of Topogonov's theorem: see \cite[Definition 4.3.1 and Theorem 10.3.1]{BBI01} for details. 

We have the following simple corollary of this characterisation of non-negative curvature.

\begin{cor}\label{topcor}
  Suppose $(X,d)$ is a complete length space with non-negative curvature.
  Let $x,y,z \in X$, let $\gr_{xy}$ and $\gr_{xz}$ be two unit speed minimising geodesics from $x$ to $y$ and $z$ respectively, and denote the angle $\angle(\gr_{xy}^\prime(0),\gr_{xz}^\prime(0))$ by $\gq$.
  Then
  \begin{equation*}
    d(y,z) \leq d(x,z) \tan \gq.
  \end{equation*}
\end{cor}

\begin{proof}
  We have
  \begin{equation*}
    \gq = \lim_{t \to 0} \angle(\gr_{xy}^\prime(t),\gr_{xz}^\prime(t)) \geq \gq^\prime
  \end{equation*}
  by Topogonov's theorem (as stated above), where $\gq^\prime$ is the comparison angle $\wtd{\angle}yxz$.
  By basic trigonometry,
  \begin{equation*}
    \tan \gq^\prime = \frac{d(y,z)}{d(x,z)},
  \end{equation*}
  and so we have
  \begin{equation*}
    \tan \gq \geq \frac{d(y,z)}{d(x,z)}.
  \end{equation*}
  This yields the result.
\end{proof}

In particular, if $\gr_1$ and $\gr_2$ are two unit speed geodesics emanating from a point $x \in X$ with $\angle(\gr_1^\prime(0),\gr_2^\prime(0)) \leq \tan^{-1}(1/4)$, then
\begin{equation*}
  d(\gr_1(t),\gr_2(t)) \leq t/4
\end{equation*}
for all $t > 0$, since $d(\gr_2(0),\gr_2(t)) \leq t$.

\subsection{Cone covering}

In this section, we assume that $X$ is a complete geometrically doubling Riemannian manifold, so that $(X,d)$ is a complete length space.
We also fix $\gfv$ and $m$ satisfying condition (A) as in Section \ref{basicsetting} and assume in addition the following
comparability condition:
\begin{description}
\item[(C)] For every $\ga > 0$, there exists a constant $c_\ga$ such that for all pairs of points $x,y \in X$,
  \begin{equation*}
    \quad d(x,y) \leq \ga m(x) \implies m(x) \leq c_\ga m(y).
  \end{equation*}
\end{description}

\begin{rmk}
  We could work in the context of complete geometrically doubling non-negatively curved length spaces; we have imposed smooth structure in order to use the language of tangent spaces rather than that of spaces of directions.
  The length space setting is only a small generalisation of the manifold setting, due to the fact that complete non-negatively curved length spaces are manifolds almost everywhere.
\end{rmk}

Given parameters $\ga\geq 1$ and $\gl \in (0,1)$, we define the extension of an open set $E \subset X$ by
\begin{equation*}
  E_{\ga,\gl}^* := \bigcup \Big\{ B\in\mc{B}_\ga : \frac{\gg(B \cap E)}{\gg(B)} > \gl \Big\} .
\end{equation*}
Note that we can write
\begin{equation*}
  E_{\ga,\gl}^* = \{x \in X : M_\ga \mb{1}_E (x) > \gl\},
\end{equation*}
where $M_\ga$ is the $\ga$-local maximal operator from Appendix \ref{lmf}, and so 
$E_{\ga,\gl}^*$ is open.
Furthermore, since for each $\ga \geq 1$ the local maximal function is of weak type $(1,1)$ with respect to $\gg$, we have
\begin{equation*}
  \gg(E_{\ga,\gl}^*) \leq \frac{C_\ga}{\gl}\gg(E)
\end{equation*}
for all $\gl \in (0,1)$.

For all $x \in X$, for all unit tangent vectors $v \in T_x X$ (recalling that we have assumed that $X$ is a manifold), and for all $t > 0$, define the \emph{sector}
\begin{equation*}
  R(v,t) := \bigcup_{0 \leq s \leq t} B(\gr (s),s/4)
\end{equation*}
opening from $x$ in the direction of $v$ along the unit speed geodesic $\gr$ with 
$\gr'(0) = v$.

\begin{lem}\label{lda}
  Let $\gb \geq 1$.
  There exists $\ga \geq 1$ and $\gl \in (0,1)$ 
  such that the following holds: if $E \subset X$ is open and 
  $y \in R(v,t) \subset E$, with $v\in T_xX$ and 
  $0 < t \leq \gb m(x)$, then $B(y,2t) \subset E_{\ga,\gl}^*$.
\end{lem}
\begin{proof}
  Suppose that $E \subset X$ is open and $y \in R(v,t) \subset E$, with $v\in T_xX$ and $0 < t \leq \gb m(x)$.
  We search for $\ga$ and $\gl$ so that 
  \begin{equation*}  
    B(y,2t)\in \mc{B}_\ga \quad \text{and} \quad
    \frac{\gg (B(y,2t)\cap E)}{\gg (B(y,2t))} > \gl .
  \end{equation*}
  Denote by $\gr$ the unit speed geodesic determined by $v$ and begin by observing that $B(\gr (t),t/4) \subset R(v,t) \subset B(y,2t)\cap E$, while $B(y,2t)\subset B(\gr (t),4t)$, so that
  \begin{equation*}
    \frac{\gamma (B(y,2t)\cap E)}{\gamma (B(y,2t))} 
    \geq \frac{\gamma (B(\gr (t),t/4))}{\gamma (B(\gr (t),4t))} .
  \end{equation*}
  Now $d(x,\gr (t)) \leq t \leq \gb m(x)$, and by (C) we have $m(x) \leq c_\gb m(\gr (t))$, so $t \leq \gb m(x)\leq \gb c_\gb m(\gr (t))$.
  This means that $B(\gr (t),t/4)$ is $\gb c_{\gb} / 4$-admissible, so that by (A),
  \begin{equation*}
    \gg (B(\gr (t),4t)) \leq A_\gb \gg \left(B\left(\gr (t),\frac{t}{4}\right)\right)
  \end{equation*}
  for some constant $A_\gb$.
  We may now choose $\gl < 1/A_\gb$ to get
  \begin{equation*}
    \frac{\gamma (B(y,2t)\cap E)}{\gamma (B(y,2t))} > \lambda .
  \end{equation*}
  To choose $\ga$, note that since $d(x,y) \leq 2t \leq 2\gb m(x)$, we have $m(x) \leq c_{2\gb}m(y)$, and so $t\leq \gb c_{2\gb} m(y)$.
  In order to have $B(y,2t)\in\mc{B}_\ga$, we choose $\ga = 2\gb c_{2\gb}$.
  By the definition of the extension, we now have $B(y,2t)\subset E^*_{\ga,\gl}$.
\end{proof}

Dictated by the last paragraph in the proof of the following lemma, we now fix $\beta = c_1$, and choose $\ga$ and $\gl$ in accordance with Lemma \ref{lda}.
We also write $E^* = E_{\ga,\gl}^*$.
Recall that the admissible tent $T(O)$ over an open set $O \subset X$ is given by
\begin{equation*}
  T(O) := D \sm \gG(O^c),
\end{equation*}
where $\gG(O^c) := \cup_{x \in O^c} \gG(x)$.

\begin{lem}[Cone covering lemma]\label{cclem}
  Assume that $X$ is non-negatively curved, and let $E \subset X$ be a bounded open set.
  Then for every $x \in E$ there exist finitely many points $x_1, \ldots, x_N \in X\sm E$, with $N$ depending only on the dimension of $X$, such that
  \begin{equation*}
    \gG(x) \sm T(E^*) \subset \bigcup_{m=1}^N \gG(x_m).
  \end{equation*}
\end{lem}
\begin{proof}
  Let $x\in E$ and pick unit vectors $v_1,\ldots , v_N \in T_xX$ so that every $v\in T_xX$ has $\angle (v,v_m) \leq \tan^{-1} (1/4)$ for some $m=1,\ldots , N$.
  For each $m$, denote by $\gr_m$ the unit speed geodesic determined by $v_m$,
  and let $t_m > 0$ be the minimal number ($E$ is bounded) for which 
  $\overline{B}(\gr_m (t_m),t_m/4)$ intersects $X\sm E$, so that we may choose an
  $x_m\in (X\sm E) \cap \overline{B}(\gr_m (t_m),t_m/4)$. Note that now
  $R(v_m,t_m)\subset E$ for each $m$.   
  
  Letting $(y,t)\in\Gamma (x) \sm T(E^*)$, we need to show that $d(y,x_m) < t$ for some
  $m$. By completeness of $X$, we may choose a unit speed minimising 
  geodesic $\gr$ from $x$ to $y$ and then
  fix an $m$ so that $\angle (\gr'(0), v_m) \leq \tan^{-1}(1/4)$. Corollary \ref{topcor}
  guarantees that $y \in R(v_m,d(x,y))$.
  
  Suppose first that $x$ is close to $E^c$ in the direction of $v_m$, in the sense that
  $t_m \leq \beta m(x)$. If $d(x,y) > t_m$, then by Corollary \ref{topcor}
  $\gr (t_m)$ is in $\overline{B}(\gr_m (t_m) , t_m/4)$, and so
  \begin{align*}
    d(y,x_m) &\leq d(y,\gr(t_m)) + d(\gr(t_m),x_m) \\ 
    &\leq d(y,\gr(t_m)) + \frac{t_m}{2} \\
    &\leq d(y,\gr(t_m)) + d(\gr(t_m),x) \\
    &= d(y,x) < t .
  \end{align*} 
  On the other hand,
  if $d(x,y)\leq t_m$, then $y\in R(v_m,t_m)$---that is, 
  $y\in B(\gr_m(s) , s/4)$ for some $0\leq s \leq t_m$---and so 
  \begin{align*}
    d(y,x_m) &\leq d(y,\gr_m(s)) + d(\gr_m(s), \gr_m(t_m)) + d(\gr_m(t_m) , x_m) \\
    &\leq \frac{s}{4} + t_m - s + \frac{t_m}{4} \leq 2t_m .
  \end{align*}    
  According to Lemma \ref{lda}, $B(y,2t_m)\subset E^*$, but since 
  $(y,t)\not\in T(E^*)$ implies that $B(y,t)\not\subset E^*$, we must have $2t_m < t$.
  
  Second, we show that it is not possible to have $t_m > \gb m(x)$ with $\gb = c_1$.
  Note first that since $d(x,y) < t < m(y)$, we have by (A1) that $t < m(y) \leq c_1 m(x)$.
  If indeed we had $t_m > c_1 m(x)$, then
  $y\in R(v_m,c_1 m(x)) \subset R(v_m,t_m) \subset E$. Invoking Lemma \ref{lda} gives
  $B(y,c_1 m(x))\subset B(y,2c_1 m(x))\subset E^*$, 
  while $B(y,t)\not\subset E^*$ and so $c_1 m(x) < t$, which is a contradiction.
\end{proof}

The cone covering lemma allows stronger pointwise estimation of the functional $\mc{A}_q$ when $q \geq 1$ (cf.
Lemma \ref{pointwise2}):

\begin{cor}
  \label{pointwise}
  Assume that $X$ is non-negatively curved. 
  Suppose $1 \leq q < \infty$, and let $f$ be a function on $D$ with bounded support. Let $\gl > 0$ and write
  $E = \{ x\in X : \mc{A}_q f(x) > \gl \}$. Then
  \begin{equation*}
    \mc{A}_q (f\mb{1}_{D\sm T(E^*)})(x) \lesssim_{\dim X} \gl \quad \text{for all $x\in X$} .
  \end{equation*}
  \begin{proof}
    If $x\in X\sm E$, then
    \begin{equation*}
      \mc{A}_q (f\mb{1}_{D\sm T(E^*)})(x) \leq \mc{A}_q f(x) \leq \gl
    \end{equation*}
    by the definition of $E$.
    So let $x\in E$.
    Since $E$ is a bounded open set, we may use Lemma \ref{cclem} to pick
    $x_1,\ldots , x_N\in X\sm E$ (with $N$ depending only on the dimension of $X$) such that
    \begin{equation*}
      \Gamma (x) \setminus T(E^*) \subset \bigcup_{m=1}^N \Gamma (x_m) .
    \end{equation*}
    We can then estimate
    \begin{align*}
      \mc{A}_q (f\mb{1}_{D\sm T(E^*)})(x)
      &= \left( \iint_{\gG (x) \sm T(E^*)} |f(y,t)|^q \frac{d\gg(y)}{\gg(B(y,t))}\frac{dt}{t} \right)^{1/q}\\
      &\leq \sum_{m=1}^N \left( \iint_{\gG (x_m)} |f(y,t)|^q \frac{d\gg (y)}{\gg(B(y,t))} \frac{dt}{t} \right)^{1/q} 
      \leq N\gl,
    \end{align*}
    proving the corollary.
  \end{proof}
\end{cor}

\begin{rmk}
  At the time of writing we do not know of any doubling Riemannian manifolds (equipped with $\gfv$ and $m$) for which the cone covering lemma fails.
  It would be interesting to determine more precisely which spaces admit cone coverings of the type above.
\end{rmk}

\end{appendix}

\footnotesize
\bibliographystyle{amsplain}
\bibliography{analysis}  

\providecommand{\bysame}{\leavevmode\hbox to3em{\hrulefill}\thinspace}
\providecommand{\MR}{\relax\ifhmode\unskip\space\fi MR }
\providecommand{\MRhref}[2]{%
  \href{http://www.ams.org/mathscinet-getitem?mr=#1}{#2}
}
\providecommand{\href}[2]{#2}
\begin{thebibliography}{10}

\bibitem{aA13}
A.~Amenta, \emph{Tent spaces over metric measure spaces under doubling and
  related assumptions}, Operator Theory in Harmonic and Non-commutative
  Analysis: Proceedings of IWOTA 2012 (to appear).

\bibitem{pA11}
P.~Auscher, \emph{Change of angle in tent spaces}, C. R. Math. Acad. Sci. Paris
  \textbf{349} (2011), no.~5-6, 297--301.

\bibitem{AMR08}
P.~Auscher, A.~McIntosh, and E.~Russ, \emph{{Hardy} spaces of differential
  forms on {Riemannian} manifolds}, J. Geom. Anal. \textbf{18} (2008),
  192--248.

\bibitem{dB87}
D.~Bakry, \emph{\'{E}tude des transformations de {R}iesz dans les vari\'et\'es
  riemanniennes \`a courbure de {R}icci minor\'ee}, S\'eminaire de
  {P}robabilit\'es, {XXI}, Lecture Notes in Math., vol. 1247, Springer, Berlin,
  1987, pp.~137--172.

\bibitem{aB92}
A.~Bernal, \emph{Some results on complex interpolation of {$T^p_q$} spaces},
  Interpolation spaces and related topics (Ramat-Gan), Israel Mathematical
  Conference Proceedings, vol.~5, 1992, pp.~1--10.

\bibitem{BBI01}
D.~Burago, I.~Burago, and S.~Ivanov, \emph{A course in metric geometry},
  Gradute Studies in Mathematics, vol.~33, American Mathematical Society,
  Providence, 2001.

\bibitem{CMM09}
A.~Carbonaro, G.~Mauceri, and S.~Meda, \emph{{$H^1$} and {BMO} for certain
  locally doubling metric measure spaces}, Ann. Sc. Norm. Super. Pisa Cl. Sci.
  (5) \textbf{8} (2009), no.~3, 543--582.

\bibitem{CMM10}
\bysame, \emph{{$H^1$} and {BMO} for certain locally doubling metric measure
  spaces of finite measure}, Colloq. Math. \textbf{118} (2010), no.~1, 13--41.

\bibitem{CMM13}
A.~Carbonaro, A.~McIntosh, and A.~J. Morris, \emph{Local {H}ardy spaces of
  differential forms on {R}iemannian manifolds}, J. Geom. Anal. \textbf{23}
  (2013), no.~1, 106--169.

\bibitem{CMS85}
R.~R. Coifman, Y.~Meyer, and E.~M. Stein, \emph{Some new function spaces and
  their applications to harmonic analysis}, J. Funct. Anal. \textbf{62} (1985),
  304--335.

\bibitem{GPLMR01}
P.~G{\'o}mez~Palacio, J.~A. L{\'o}pez~Molina, and M.~J. Rivera, \emph{Some
  applications of the lattice finite representability in spaces of measurable
  functions}, Turkish J. Math. \textbf{25} (2001), no.~4, 475--490.

\bibitem{HTV91}
E.~Harboure, J.~L. Torrea, and B.~Viviani, \emph{A vector-valued approach to
  tent spaces}, J. Anal. Math. \textbf{56} (1991), 125--140.

\bibitem{HK12}
T.~Hyt{\"o}nen and A.~Kairema, \emph{Systems of dyadic cubes in a doubling
  metric space}, Colloq. Math. \textbf{126} (2012), no.~1, 1--33.

\bibitem{mK14}
M.~Kemppainen, \emph{On vector-valued tent spaces and {H}ardy spaces associated
  with non-negative self-adjoint operators}, Preprint (2014), arXiv:1402.2886.

\bibitem{mK11}
\bysame, \emph{The vector-valued tent spaces {$T^1$} and {$T^\infty$}}, J.
  Aust. Math. Soc. (to appear).

\bibitem{MvNP11}
J.~Maas, J.~van Neerven, and P.~Portal, \emph{Conical square functions and
  non-tangential maximal functions with respect to the {G}aussian measure},
  Publ. Mat. \textbf{55} (2011), 313--341.

\bibitem{MvNP12}
\bysame, \emph{Whitney coverings and the tent spaces {$T^{1,q}(\gamma)$} for
  the {G}aussian measure}, Ark. Mat. \textbf{50} (2012), no.~2, 379--395.

\bibitem{MT00}
T.~Mart{\'{\i}}nez and J.~L. Torrea, \emph{Operator-valued martingale
  transforms}, Tohoku Math. J. (2) \textbf{52} (2000), no.~3, 449--474.

\bibitem{MM07}
G.~Mauceri and S.~Meda, \emph{{${\rm BMO}$} and {$H^1$} for the
  {O}rnstein-{U}hlenbeck operator}, J. Funct. Anal. \textbf{252} (2007), no.~1,
  278--313.

\bibitem{MMS12}
G.~Mauceri, S.~Meda, and P.~Sj{\"o}gren, \emph{Endpoint estimates for
  first-order {R}iesz transforms associated to the {O}rnstein-{U}hlenbeck
  operator}, Rev. Mat. Iberoam. \textbf{28} (2012), no.~1, 77--91.

\bibitem{tM03}
T.~Mei, \emph{{BMO} is the intersection of two translates of dyadic {BMO}}, C.
  R. Math. \textbf{336} (2003), no.~12, 1003--1006.

\bibitem{pM84}
P.-A. Meyer, \emph{Transformations de {R}iesz pour les lois gaussiennes},
  Seminar on probability, {XVIII}, Lecture Notes in Math., vol. 1059, Springer,
  Berlin, 1984, pp.~179--193.

\bibitem{PPpc}
P.~Portal, \emph{Personal communication}.

\bibitem{pP13}
\bysame, \emph{Maximal and quadratic {G}aussian {H}ardy spaces}, Rev. Mat.
  Iberoam. \textbf{30} (2014), no.~1, 79--108.

\bibitem{RdF86}
J.~L. Rubio~de Francia, \emph{Martingale and integral transforms of {B}anach
  space valued functions}, Probability and {B}anach spaces ({Z}aragoza, 1985),
  Lecture Notes in Math., vol. 1221, Springer, Berlin, 1986, pp.~195--222.

\bibitem{emS70}
E.~M. Stein, \emph{Topics in harmonic analysis related to the
  {L}ittlewood-{P}aley theory.}, Annals of Mathematics Studies, No. 63,
  Princeton University Press, Princeton, N.J., 1970.

\bibitem{hT78}
H.~Triebel, \emph{Interpolation theory, function spaces, differential
  operators}, North-Holland Mathematical Library, vol.~18, North-Holland
  Publishing Company, Amsterdam, 1978.

\bibitem{dW91}
D.~Williams, \emph{Probability with martingales}, Cambridge Mathematical
  Textbooks, Cambridge University Press, Cambridge, 1991.

\end{thebibliography}
\end{document}